\documentclass[11pt]{amsart}
\usepackage{geometry}                
\usepackage{graphicx}
\usepackage{epstopdf}
\usepackage{amsmath,amssymb,latexsym,pstricks,mathrsfs,comment,amsthm,graphicx,tikz,enumerate,pgffor,cite,wrapfig,float}
\usepackage[modulo,pagewise]{lineno}
\numberwithin{equation}{section}

\newtheorem{theorem}{Theorem}[section]

\newtheorem{corollary}[theorem]{Corollary}

\newtheorem{definition}[theorem]{Definition}
\newtheorem{definitions}[theorem]{Definitions}

\newtheorem{lemma}[theorem]{Lemma}

\newtheorem{proposition}[theorem]{Proposition}
\newtheorem{remark}[theorem]{Remark}

\newcommand{\nc}{\newcommand}
\nc{\rnc}{\renewcommand}
\nc{\Rho}{{\rm P}}
\nc{\ip}{idempotent}
\nc{\Ip}{Idempotent}
\nc{\sgp}{semigroup}
\nc{\alg}{algebra}
\nc{\sla}{semilattice}
\nc{\ioi}{if and only if~}
\nc{\cs}{completely simple}
\nc{\ch}{character}
\nc{\supp}{\operatorname{supp}}

\title[Trace- and pseudo-products]{Trace- and pseudo-products: Restriction-like semigroups with a band of projections}
\author{D. G. FitzGerald  \& M. K. Kinyon }
\address{School of Natural Sciences, University of Tasmania, POB 37 Hobart, Australia 7001}
\address{Department of Mathematics, 
University of Denver, 
Denver CO 80208}
\date{\today}                                          

\begin{document}
\maketitle 

\begin{abstract}
We ascertain conditions and structures on categories and semigroups which admit the construction of pseudo-products and trace products respectively, making their connection as precise as possible.  This topic is modelled on the ESN Theorem and its generalisation to ample semigroups.  Unlike some other variants of ESN, it is self-dual (two-sided), and the condition of commuting projections is relaxed. The condition that projections form a band (are closed under multiplication) is shown to be a very natural one.  One-sided reducts are considered, and compared to (generalised) D-semigroups.  
Finally the special case when the category is a groupoid is examined.  
\end{abstract}
\section{Introduction} 

The distant aim of this paper is to enrich the study of a groupoid 
 over a skew lattice of identities, begun by the first author in \cite{Fi}.  But since a skew lattice possesses two operations, each of which confers a band structure,  a necessary step is to study a groupoid, or more generally a category, over a band of objects.   That is the more immediate aim here, and locates the paper squarely in the tradition of the Ehresmann-Schein-Nambooripad (ESN) Theorem, which (in its various forms) asserts equivalence of certain kinds of categories having extra structure with certain special kinds of semigroups.  This  equivalence can help elucidate one kind of structure in terms of the other; for example, the ESN Theorem has recently been used to design a relatively efficient enumeration of finite inverse semigroups \cite{Ma}.  
We must take a little further space to summarise the location of the present paper in this extensive field. 
 
Ehresmann's use of  pseudogroups in describing symmetries in differential geometry, and Wagner's and Schein's similar use of groupoids and inverse semigroups, are described by Lawson in Sections 1.1--3 of the book \cite{LaIS}, with a history of the ideas behind their respective contributions outlined in Section 4.4 of the same.  
The interested reader is referred to those sections, to the extended account of Hollings  \cite{Ho2}, or to the  original works cited therein.  
Suffice it to say here that the constructions of the trace product in a suitable semigroup, and the pseudoproduct in a suitable groupoid, are central to the success of the Theorem.

Nambooripad's concern was with regular semigroups, and the constructions in his works are more general and more complex; see the survey article of Szendrei \cite{Sz}, especially Section 4.  His point of departure was the description by Miller and Clifford of a completely $0$-simple (regular) semigroup as a Rees groupoid \cite {MiCl}; he described and formulated the notion of a bi-ordered set (characterising the set of idempotents, and initiating a whole field of research which continues today).  This he used in conjunction with a particular kind of ordered groupoid generalising Rees groupoids (the inductive groupoids in his terminology---a different use of the term than elsewhere), to formulate his extension of  the Theorem, incidentally making it fully categorical in flavour.
 Lawson \cite{La91} later extended the use of ordered groupoids to ordered (small) categories, and this kind of generalisation is also followed in the paper to hand.

All these authors were concerned to explore the conditions which made the correspondence work; of course there are many ``answers'' possible.  This paper has the declared aim of tackling this question in the context of the idempotents or projections forming a band (it will be seen that this  condition arises naturally). 
Thus the context here is broader than the semilattice of idempotents of Ehresmann and Schein, but not as broad as Nambooripad's bi-ordered sets.  
The semigroups we consider are a particular class of the $P$-semiabundant semigroups of Section 1  in Lawson \cite{La91}, but distinct from the Ehresmann semigroups of Section 3, which are the main focus of that paper.   
   
It was realised, particularly in work of the `York school',  that the field overlaps (and has been extended to) other semigroup properties, some of which arose in the theory of semigroup acts; notably, it was refreshed and organised by Lawson.  
Its ramifications and notation have been rationalised in the unpublished but very influential {\it Notes} of Gould \cite{GoN};   Hollings \cite{Ho} surveyed its reach and history.  Again the reader seeking more detail is referred to these surveys.  
Most recently, contributions relevant to the present paper have been made by Wang \cite{Wa}, Kudryavtseva \cite{Ku} and El-Qallali \cite{Q}. 
Wang's paper \cite{Wa}---which includes a useful summary of the several versions of ESN equivalences mentioned above---is the key part of a project designed to achieve an  ESN equivalence for (a generalisation of) orthodox semigroups.  The present paper begins, in contrast, at the category side.  We thus circumvent or postpone some of the arcane semigroup properties otherwise required, such as generalised Green's relations and congruence properties.  When specialising to groupoids, however, we reach a similar endpoint to Wang's, and the relationship between our approaches is yet to be explored.

Jones's paper \cite{Jo} on restriction semigroups is close in spirit to ours on the semigroup side, though we also consider equivalent categories.  Jones relativises the Ehresmann-style axioms to a specified set  $P$ of commuting projections; in a generalisation, he includes in his scheme the regular $\ast$-semigroups of Nordahl and Scheiblich \cite{NS}, in which $P$ is generally not closed under multiplication.  More recently, Branco et al. \cite{BGG} and Lawson \cite{La21} have focused on the class of Ehresmann semigroups.  In contrast to all these, our $P$ is always a band.

An important distinction among all these studies is between one-sided and two-sided conditions.  Gould and Stokes \cite{GoSt, St1, St2} have addressed corresponding questions for the one-sided case, in a strand which goes back to Cockett and Lack's \emph{restriction categories}  \cite{CoLa}.  
Because of our concern with categories and the ESN theorem, our approach is mainly two-sided, though we are able also to make a small contribution in the one-sided theory in Sections 4 and 6 below.  
All that said, in the end we aim to understand the structures that arise when a band or a skew lattice is extended in a manner analogous with inverse semigroup extensions of a semilattice---or a lattice, as in boolean inverse semigroups.
\section{Small categories}
The system $(C, \circ, +, - )$, where $C$ is a set, $\circ$ a partially defined binary operation $C\times C \rightarrow C$, and $+, -$ are unary maps $C\rightarrow C$, is a (small) \emph{category} if it satisfies the axioms 
\begin{subequations} \label{eq:cat}
\begin{align}
  &x^{+}\circ x = x = x\circ x^{-};\\
  & (x^{+})^{-} = x^{+}, (x^{-})^{+} = x^{-};\\
   & x\circ y \text{   is defined precisely when }x^{-} = y^{+};\\
 &  \text{   when  } x^{-} = y^{+}, \hspace{6pt} (x\circ y)^{+} = x^{+} \text{      and     } (x\circ y)^{-} = y^{-};\\
 & (x\circ y)\circ z = x\circ (y\circ z) \text{   whenever the products are defined}.
  \end{align}
\end{subequations}
The set $C^{+} = \{e \in C\colon e=e^{+}\} = C^{-}$ is the set of \emph{identities} or \emph{objects}.  Axiom (\ref{eq:cat}b) implies 
$(x^{+})^{+} =  ((x^{+})^{-})^{+}  = (x^{+})^{-} = x^{+}$ and similarly or dually, $(x^{-})^{-} = x^{-}$.
For $e\in C^{+}$ and $x\in C$, 
$e\circ x$ is defined if and only if $x^{+} = e^{-} = e$ and then 
$e\circ x = x$; and a dual statement holds for $x\circ e$. 
  
\section {Transcription categories}
A \emph{transcription category} is a category as above, endowed with two more maps $C\times C^{+} \rightarrow C$, called \emph{transcription maps} and written as follows: for $e,f \in C^{+}$ and $x\in C$, 
$_{e} | x , x|_{f}$ denote arrows (members of $C$) satisfying the axioms below, which are based on those for ordered or inductive categories.  Unlike some other studies, we do not impose any order conditions on the objects to which $x\in C$ may be `restricted': this is partly why we have used the term `transcription'.
\begin{subequations} \label{eq:skewcat}
\begin{align}
&  _{e}|f = e|_{f};\\
 & _{x^{+}}|x = x = x|_{x^{-}}; \\
&  _{e}|(_{f}|x) =\; _{(_{e}|f)}|x \text{    and its dual    } x|_{(e|_{f})} = (x|_{e})|_{f};\\
&  _{e}|(x\circ y) = (_{e}|x)\circ(_{(_{e}|x)^{-}}|y)\text {    and its dual    } (x\circ y)|_{f} = (x|_{(y|_{f})^{+}})\circ y|_{f};\\
 & (_{e}|x)^{+} =\; _{e}|x^{+} \text{    and its dual    } (x|_{f})^{-} = x^{-}|_{f};\\
 &  (_{e}|x)|_{f} = \; _{e}|(x|_{f}).
 \end{align}
\end{subequations}
Axiom (\ref{eq:skewcat}a) ensures the two maps agree on $C^{+}\times C^{+}$ and by (\ref{eq:skewcat}e) map $C^{+}\times C^{+}\rightarrow C^{+}$. 
Axiom (\ref{eq:skewcat}b) implies $_{e}|e= e$ and so with associativity from (\ref{eq:skewcat}c) we have  that $C^{+}$ is a band under the operation $| \,\colon\, (e,f)\mapsto e\vert_{f}$.  Where convenient, we may write $e|_{f}$ as $e|f$ because of (\ref{eq:skewcat}a).  The band $C^{+}$ acts on the set $C$ on both left and right by (\ref{eq:skewcat}c), interacting with the category operations  as per (\ref{eq:skewcat}d, e), and the two actions are linked by (\ref{eq:skewcat}f).
 
Observe that the natural-looking equation  
$  (_{e}|x)^{-} = \;_{(_{e}|x)^{-}}|x^{-} = (_{e}|x)^{-}|_{x^{-}}$  and its dual  $ (x|_{f})^{+} =  x^{+}|_{(x|_{f})^{+}} = \;_{x^{+}}|(x|_{f})^{+}$ 
are consequences of certain of these axioms: put $y=x^{-}$ in (\ref{eq:skewcat}d), so the compositions are defined and  $_{e}\vert x = {_{e}}\vert (x\circ x^{-}) {=}(_{e}\vert x)\circ (_{(_{e}\vert x)^{-}}\vert x^{-})$, whence 
\begin{equation}\label{oddone}
(_{e}\vert x)^{-} \overset{\ref{eq:cat}d}{=} (_{(_{e}\vert x)^{-}}\vert x^{-})^{-}
\overset{\ref{eq:skewcat}a}{=}{} (_{e}\vert x)^{-}\vert _{x^{-}}. 
\end{equation}

Given a transcription category $(C, \circ, +,-)$ or $C$ for short, we define a \emph{pseudoproduct} $\otimes$ as follows.  For $x,y \in C$, define 
\begin{equation}\label{defotimes}
 x\otimes y = (x|_{y^{+}})\circ (_{x^{-}}|y), 
\end{equation}
the right hand side being defined in ${C}$ since 
$$(x|_{y^{+}})^{-} =x^{-}|_{y^{+}} = \:_{x^{-}}|y^{+} = (\;_{x^{-}}|y)^{+}$$
by (\ref{eq:skewcat}e) and (\ref{eq:skewcat}a).  
The following properties will be required.  Note that $\otimes$ extends both $\circ$ and the restriction operations $|$ when they are defined.
\begin{lemma}\label{lem}
\begin{enumerate}[i)]
\item If $x\circ y$ is defined then  $x\otimes y = x\circ y$; 
\item if $e\in C^{+}$ and $x\in C$,  $e\otimes x = \,_{e}|x$ and  $x\otimes e = x|_{e}$;  
\item $(x\otimes y)^{+} = (x|_{y^{+}})^{+}$ and $(x\otimes y)^{-} = (_{x^{-}}|y)^{-}$;
\item if $e\in C^{+}$ and $x,y \in C$, then $\,_{e}|(x\otimes y) = (\,_{e}|x)\otimes y$;
\item $x\otimes y = (x|_{x^{-}y^{+}})\circ(\,_{x^{-}y^{+}}|y)$.
\end{enumerate}
 \end{lemma}
\begin{proof}
(i) If $x^{-} = y^{+}$, then $x|_{y^{+}} = x|_{x^{-}} = x$,  by (\ref{eq:skewcat}b), and dually $_{x^{-}}|y = y$; the result follows from the definition in equation (\ref{defotimes}).  

(ii) We have $e\otimes x \overset{\text{def.}}{=} (e|_{x^{+}})\circ(\, _{e^{-}}|x) \overset{\ref{eq:skewcat}a}{=} (\,_{e}|x^{+})\circ(\,_{e}|x)  \overset{\ref{eq:skewcat}e}{=} (\,_{e}|x)^{+} \circ (\,_{e}|x) = (\,_{e}|x)$, and dually.

(iii) This comes from the definition and (\ref{eq:cat}d).

(iv) The left-hand expression is $\,_{e}|(x\otimes y) = \,_{e}|((x|_{y^{+}})\circ (\,_{x^{-}}|y)) \overset{\ref{eq:skewcat}d}{=} (\,_{e}|x|_{y^{+}})\circ (_{(\,_{e}|x|_{y^{+}})^{-}}|(\,_{x^{-}}|y))$,  where the second factor simplifies, using (\ref{eq:skewcat}c), as follows: $$\,_{(\,_{e}|x)^{-}x^{-}y^{+}x^{-}y^{+}}|y  \overset{\text{band}}{=}    \,_{(\,_{e}|x)^{-}x^{-}y^{+}}|y =   \,_{(\,_{e}|x)^{-}x^{-}}|y \overset{\ref{oddone}} {=} \,_{(\,_{e}|x)^{-}}|y .$$
The right-hand expression is $(\,_{e}|x)|_{y^{+}} \circ (\,_{(\,_{e}|x)^{-}}|y)$, and the two are equal.

(v) It is enough to note that $x|_{y^{+}} = (x|_{x^{-}})|_{y^{+}}$, etc., and use axiom (\ref{eq:skewcat}c) and the definition (\ref{defotimes}).
\end{proof}
\begin{theorem} \label{is sgp}
 If ${C}$ is a transcription category then $(C,\otimes)$ is a semigroup.
\end{theorem}
\begin{proof}
We have to show associativity of $\otimes$.  
 First, for any $x,y,z\in C$ such that $x\circ y$ is defined,
\begin{align*}
& (x\circ y)\otimes z   &\overset{\ref{defotimes}}{=} & &(x\circ y)|_{z^{+}} \circ (_{(x\circ y)^{-}}| z)  &&\overset{\ref{eq:skewcat}d}{=}& & (x|_{(y|_{z^{+}})^{+}}\circ y|_{z^{+}}) \circ (_{(x\circ y)^{-}}| z)\\
&&\overset{\ref{eq:cat}e}{=} &&x|_{(y|_{z^{+}})^{+}}\circ (y|_{z^{+}} \circ \:_{(x\circ y)^{-}}| z) && \overset{\ref{eq:cat}d}{=} & &x|_{(y|_{z^{+}})^{+}}\circ (y|_{z^{+}} \circ \:_{y^{-}}| z) \\
&&\overset{\ref{defotimes}}{=}&& x|_{(y|_{z^{+}})^{+}}\circ (y\otimes z) && = &&x\otimes (y\otimes z),
\end{align*}
since $x^{-} = y^{+}$ and $(y\otimes z)^{+} = (y|_{z^{+}})^{+}$ by Lemma \ref{lem}(iii).
Next, in the general case we have 
\begin{align*}
&(x\otimes y)\otimes z &\overset{\ref{defotimes}}{=}& &(x|_{y^{+}} \circ \;_{x^{-}}|y) \otimes z &\overset{\text{above}}{=}& & x|_{y^{+}}\otimes ((\,_{x^{-}}|y)\otimes z) \overset{\text{Lem}\ref{lem}\text{(iv)}}{=}  x|_{y^{+}}\otimes \,_{x^{-}}|(y\otimes z) .
\end{align*}
Consider the last expression above: we know $(x|_{y^{+}})^{-} = x^{-}|y^{+}$ by (\ref{eq:skewcat}e) and $$(\,_{x^{-}}|y\otimes z)^{+} = x^{-}|(y\otimes z)^{+} = x^{-}|(y|_{z^{+}})^{+}$$ by Lemma \ref{lem}(iii), so 
\begin{align*}
 x|_{y^{+}}\otimes \,_{x^{-}}|(y\otimes z) && {=} && ((x|_{y^{+}})|_{x^{-}|(y|_{z^{+}})^{+}})\circ (\,_{x^{-}|y^{+}|x^{-}}|(y\otimes z)).
\end{align*}
Now the left factor is
\begin{align*}
&&(x|_{y^{+}})|_{x^{-}|(y|_{z^{+}})^{+}} && = && x|_{(x^{-}|y^{+})(x^{-}|y^{+})|(y|_{z^{+}})^{+}}&& = && x|_{(x^{-}|y^{+})(y|_{z^{+}})^{+}} \\
=&& x|_{x^{-}|(y^{+}|(y|_{z^{+}})^{+})} &&=&& x|_{(y|_{z^{+}})^{+}}&&=&&x|_{(y\otimes z)^{+}} ,
\end{align*}
 the right factor is
\begin{align*}
 \,_{x^{-}|y^{+}|x^{-}|y^{+}}|(y\otimes z) &&=&& \,_{x^{-}|y^{+}}|(y\otimes z)  &&=&& \,_{x^{-}}|(\,_{y^{+}}|y\otimes z)&&=&&\,_{x^{-}}|(y\otimes z),
  \end{align*}
and finally we have $(x\otimes y)\otimes z = x\otimes (y\otimes z)$.

\end{proof}
The semigroup $(C,\otimes)$ is denoted $\mathscr{S}(C)$ and has extra structure, which we proceed to explore.

\section{localisable semigroups}
A \emph{unary semigroup} is an algebra $(S,\cdot, +)$ of signature $(2,1)$ such that $(S,\cdot)$ is a semigroup, that is, $x\cdot(y\cdot z) = (x\cdot y)\cdot z$,  together with a unary map $+\colon S\rightarrow S$. 
A \emph{left localisable semigroup} is a unary semigroup  
such that the unary map $+\colon S\rightarrow S$ satisfies
\begin{subequations} \label{eq:locsgp}
\begin{align}
 x^{+}\cdot x &= x , \\ 
  (x\cdot y)^{+} &= (x\cdot y^{+})^{+} ,  \text{     and} \\ 
  x^{+}\cdot y^{+} &= (x^{+}\cdot y)^{+}.   
\end{align}

A \emph{right localisable} semigroup is an algebra $(S,\cdot, -)$  such that $(S^{\rm opp},\cdot, -)$ is a left localisable semigroup; that is, $(S,\cdot, -)$ satisfies  the identities 
\begin{align}
   x\cdot x^{-} &= x , \\
  (x\cdot y)^{-} &=( x^{-}\cdot y)^{-},  \text{     and}\\
  x^{-}\cdot y^{-} &= (x\cdot y^{-})^{-} .
 \end{align}
A \emph{localisable} semigroup is an algebra $(S,\cdot, +, - )$ of signature $(2,1, 1)$ such that $(S,\cdot, +)$ is a left localisable and $(S,\cdot, -)$ a right localisable semigroup, and the unary operations are linked by the identities
\begin{align}
 & (x^{+})^{-} = x^{+} \text{  and }  (x^{-})^{+} = x^{-} .
\end{align}
 \end{subequations}
 
 These axioms are based on those for restriction semigroups, as well as the needs of the Theorems to follow.  A more detailed comparison with restriction semigroups and one-sided variants is given in a later Section.  For future use, let us observe that these axioms occur in dual pairs; to reduce repetition, we will often leave dual statements implicit.    

 Let $S^{+} = \{x^{+}: x\in S\}$, and note  that (\ref{eq:locsgp}g)
  implies  $S^{-} = \{x^{-}\colon x\in S\} = S^{+}$ and also $(x^{+})^{+} = ((x^{+})^{-})^{+} = (x^{+})^{-} = x^{+}$.   As in restriction semigroups, members of $S^{+}$ are called \emph{projections}. 
  The next lemma records some useful relationships.   
\begin{lemma}\label{band}
 Let $(S, \cdot, +)$ be a unary semigroup.  Then 
 \begin{enumerate}[i)]
 \item (\ref{eq:locsgp}a) and c) imply $(S^{+},\cdot)$ is a band;
 \item (\ref{eq:locsgp}a), b) and c) imply $(x^{+})^{+} = x^{+}$; and
 \item if $S^{+}$ is a band, (\ref{eq:locsgp}b) implies  (\ref{eq:locsgp}c).
 \end{enumerate}
 \end{lemma}
\begin{proof}
\begin{enumerate}[i)]
\item  (\ref{eq:locsgp}c) shows that $x^{+}\cdot y^{+}\in S^{+}$; it also implies  $x^{+}\cdot x^{+} = (x^{+}\cdot x)^{+}$ which $ = x^{+}$ by (\ref{eq:locsgp}a). 
\item Now from (\ref{eq:locsgp}b) we have $y^{+} = (y^{+}\cdot y^{+})^{+} = (y^{+})^{+}$.  
\item Set $x^{+}$ for $x$  in (\ref{eq:locsgp}b) and use $(x^{+}\cdot y^{+})^{+} =  x^{+}\cdot y^{+}$ by the band condition.  
 \end{enumerate}
 \end{proof} 
This lemma also indicates that it is a natural and interesting condition that $S^{+}$ be a band (and not necessarily a semilattice). 
 It is of further  interest to point out the following relations between axioms which involve both unary operations in (\ref{eq:locsgp}).  
\begin{lemma} \label{axequiv}
 Let $(S,\cdot, +,-)$ be a semigroup with two unary operations in which  (\ref{eq:locsgp}a), d) and g) hold.   Then
\begin{enumerate}[i)]
\item  (\ref{eq:locsgp}c) and f) together are equivalent to
\begin{equation} \label{eq:altax}
(x\cdot y^{+})^{-} = x^{-}\cdot y^{+} = (x^{-}\cdot y)^{+} ;
 \text{ and} 
\end{equation}
\item if $S^{+}$ is a band, (\ref{eq:locsgp}b) and e) together imply  (\ref{eq:altax}). 
\end{enumerate}
\end{lemma}
\begin{proof}
\begin{enumerate}[i)]  
\item Set $y^{+}$ in place of $y$ in (\ref{eq:locsgp}f) and use (\ref{eq:locsgp}g) to obtain the first equation of (\ref{eq:altax}).  For the second (and dual) equation of (\ref{eq:altax}), put $x^{-}$ for $x$ in (\ref{eq:locsgp}c). 
 In the converse direction, put $y^{-}$ for $y$ in the first equation of (\ref{eq:altax}) to get 
 $(x\cdot (y^{-})^{+})^{-} = x^{-}\cdot (y^{-})^{+}$, which is  (\ref{eq:locsgp}f) by (\ref{eq:locsgp}g); and  put $x^{+}$ for $x$ in the second equation of (\ref{eq:altax}) to get the dual (\ref{eq:locsgp}c). 
\item  By Lemma \ref{band}(ii) and its dual, (\ref{eq:locsgp}c) and f) hold.  By part i), (\ref{eq:altax}) follows.
\end{enumerate}
\end{proof}
\begin{lemma}\label{handy}
Let $(S,\cdot,+,-)$ be a localisable semigroup. Then for all $x,y\in S$,
\begin{enumerate}[i)]
 \item $x\cdot y = x\cdot y^{+}\cdot x^{-}\cdot y$ and
 \item $(x^{+}\cdot y)^{-}\cdot y^{-} = (x^{+}\cdot y)^{-}$.
\end{enumerate}
\end{lemma}
\begin{proof}  
\begin{enumerate}[i)]
\item Now (\ref{eq:locsgp}a) and d), taken with Lemma \ref{band}(i), show that  
\begin{equation*} \label{eq:mid-id}
 x\cdot y = x\cdot x^{-}\cdot y^{+ }\cdot y = x\cdot (x^{-}\cdot y^{+})^{2}\cdot y = x\cdot x^{-}\cdot (y^{+}\cdot x^{-})\cdot y^{+}\cdot y = x\cdot y^{+}\cdot x^{-}\cdot y.
\end{equation*}
 \item Putting $x^{+}\cdot y$ for $x$ in (\ref{eq:locsgp}f) we have 
 \begin{equation*}
 (x^{+}\cdot y)^{-}\cdot y^{-} = ((x^{+}\cdot y)\cdot y^{-})^{-}
  = (x^{+}\cdot y)^{-},
\end{equation*}
\end{enumerate}
 as required.
 \end{proof}

\begin{theorem}\label{th1}
 If ${C}$ is a transcription category then $\mathscr{S}(C)$ is a localisable semigroup.
\end{theorem}
\begin{proof}
Theorem \ref{is sgp} proved associativity of $\otimes$, and axiom (\ref{eq:locsgp}g) is (\ref{eq:cat}b). We must verify the remainder of axioms (\ref{eq:locsgp}), and it is enough to prove (\ref{eq:locsgp}a), b) and c), since the remainder are their duals.  For (\ref{eq:locsgp}a) we see 
$x^{+}\otimes x= (x^{+}|_{x^{+}})\circ(_{({x^{+}})^{-}}|x) = x^{+}\circ x = x$, using (\ref{eq:skewcat}b) and (\ref{eq:cat}a).   
For (\ref{eq:locsgp}b), axiom (\ref{eq:cat}d) gives
\begin{align*}
&(x\otimes y)^{+} &= && ((x|_{y^{+}})\circ (_{x^{-}}|y))^{+} &&= &&(x|_{y^{+}})^{+}, &&\text{     while}&\\
&(x\otimes y^{+})^{+} &= && ((x|_{y^{+}})\circ (_{x^{-}}|y^{+}))^{+} &&=&&((x|_{y^{+}})^{+}&&\text{     also}.&
\end{align*}
For (\ref{eq:locsgp}c),  
\begin{align*}
(x^{+}\otimes y)^{+}  
 = (x^{+}|_{y^{+}})^{+} = (x^{+}|_{y^{+}}) = x^{+}\otimes{y^{+}},
\end{align*}
using Lemma \ref{lem}.
\end{proof}
Given a localisable semigroup $(S,\cdot,+,-)$, define a 
 partial binary mapping or composition $\circ\colon S\times S\rightarrow S$, called the  \emph{trace product}, and transcription maps defined as follows.
 \begin{definition} \label{def:circ} 
(i)  The trace product  $x\circ y$ is defined exactly when $x^{-} = y^{+}$, and then  
$$x\circ y = x\cdot y .$$
(ii) Transcription maps for $x\in S, \, e,f\in S^{+}$ are defined by $\;_{e}|x = e\cdot x$ and $x|_{f} = x\cdot f$.
 \end{definition} 
Complementing Theorem \ref{th1} we have 
\begin{theorem}\label{th2}
If ${S}$ is a localisable semigroup, then $\mathscr{C}({S}) = (S,\circ ,+,-)$ is a transcription category.
\end{theorem}
\begin{proof}
 We first have to verify axioms (\ref{eq:cat}a--e). Associativity in $S$ guarantees (\ref{eq:cat}e).   
 By inspection, (\ref{eq:cat}a, b) and c) are already (\ref{eq:locsgp}a), d) and g) and Definition \ref{def:circ}(i).  
 Now use $ x^{-} = y^{+}$ in (\ref{eq:locsgp}b) to obtain (\ref{eq:cat}d) as 
 $(x\circ y)^{+} = (x\cdot y)^{+} = (x\cdot y^{+})^{+} = (x\cdot x^{-})^{+} = x^{+}$, and dually. 
 
 So $\mathscr{C}(S)$ is a category and next we must deal with axioms (\ref{eq:skewcat}).  Of these, (a, b, c) are immediate from Definition \ref{def:circ}(ii) and associativity of $S$. 
 For (\ref{eq:skewcat}d) consider the LHS: 
 $ _{e}|(x\circ y) = e\cdot(x\cdot y)$ with $x^{-} = y^{+}$.  Now check the composition on the RHS is valid: 
 $(_{(_{e}|x)^{-}}|y)^{+} = ((e\cdot x)^{-}\cdot y)^{+} = 
(e\cdot x)^{-}\cdot y^{+} = 
 (e\cdot x)^{-}\cdot x^{-} $,  using (\ref{eq:altax}).  
 By Lemma \ref{handy}(ii), this last equals $  (e\cdot x)^{-}$.  
Thus the composite $(_{e}|x)\circ(_{(_{e}|x)^{-}}|y)$ is defined, and equals $(e\cdot x)\cdot (e\cdot x)^{-}\cdot y = e\cdot x \cdot y$, so (\ref{eq:skewcat}d) is satisfied.   Wrapping up, (\ref{eq:skewcat}e) is immediate from (\ref{eq:locsgp}c), and (\ref{eq:skewcat}f) requires only associativity of ${S}$.
\end{proof}
Thus ${C}$ and ${S}$ are equivalent; in fact, 
\begin{theorem}\label{equiv}
 $\mathscr{C}(\mathscr{S}(C)) = C$  and  $\mathscr{S}(\mathscr{C}(S)) = S$.
\end{theorem}
\begin{proof}
 The constructions use the same base set $S$ and maps $+,-$. Moreover, $\circ$ is a restriction of $\cdot$, and $\otimes$ and $\cdot$ may be identified: for $x,y\in S$, $x\cdot y = x\cdot y^{+}\cdot x^{-}\cdot y$ by Lemma \ref{handy}(i).  Here, 
 $(x\cdot y^{+})^{-} = (x^{-}\cdot y)^{+}$ by (\ref{eq:altax}), whence $x\cdot y = (x\cdot y^{+})\circ(x^{-}\cdot y) = (x\vert _{y^{+}})\circ(_{x^{-}}\vert y)$ by Definition \ref{def:circ}, which equals $x\otimes y$ by (\ref{defotimes}).
\end{proof}
\subsection*{Morphisms}  Let $(S,\cdot,+,-)$ and 
$(S',\cdot,+,-)$ (in abuse of notation by overloading the operation symbols) be  localisable semigroups.  
A map $\phi\colon S\rightarrow S'$ is a \emph{morphism of  localisable semigroups} or a \emph{$\pm$-morphism} if it preserves the operations in the usual sense, i.e., it is a semigroup morphism such that $(x\phi)^{+} = (x^{+})\phi$ and $(x\phi)^{-} = (x^{-})\phi$.   
Similarly, let $(C,\circ,+,-)$ and 
$(C',\circ,+,-)$ be transcription categories.   
A map $\psi\colon C\rightarrow C'$ is a functor (morphism) of  transcription categories if it is a (small) functor in the usual sense (thus 
$(x\psi)^{+} = (x^{+})\psi$, etc.) and preserves the transcription maps, i.e., $(x\vert_{e})\psi = (x\psi)\vert_{e\psi}$, etc.  
\begin{theorem}
A map $\phi\colon S\rightarrow S'$ is a morphism of  localisable semigroups  if and only if it is a functor of  transcription categories $\mathscr{C}(S)\rightarrow \mathscr{C}(S')$.
\end{theorem}
\begin{proof}
Let $\phi\colon S \rightarrow S'$ be a morphism, so that $(x\phi)^{+} = (x^{+})\phi$ in $ \mathscr{C}(S')$.  If $x\circ y$ is defined in $\mathscr{C}(S')$, then $(x\phi)^{-} = x^{-}\phi = y^{+}\phi = (y\phi)^{+}$, so that $x\phi \circ y\phi$ is defined and $= x\phi \cdot y\phi = (x\cdot y)\phi =(x \circ y)\phi$.   If $e\in S^{+}$, 
$(x\vert_{e})\phi = (x\cdot e)\phi = x\phi \cdot e\phi = (x\phi\vert_{e\phi})$
  and similarly for the right transcription maps.  Conversely, if $\phi\colon C\rightarrow C'$ is a functor,  
  $(x\phi)^{\pm} = (x^{\pm})\phi$ and
  \begin{align*}
  (x\cdot y)\phi &&=&& (x\otimes y)\phi &&=&& [(x|_{y^{+}})\circ (_{x^{-}}|y)]\phi& &=&&  (x|_{y^{+}})\phi \circ (_{x^{-}}|y)\phi \\
  &&=&& (x\phi\vert_{y^{+}\phi}) \circ (_{x^{-}\phi}|y\phi) &&=&& x\phi \otimes y\phi &&=&& x\phi \cdot y\phi, 
 \end{align*}
 whence $\phi$ is a morphism $S\rightarrow S'$.
\end{proof}
Corresponding to the definition of morphisms, we make the following 
\begin{definition}\label{loco}
A \emph{localisable congruence} $\theta$ on the  localisable semigroup $(S,\cdot,+,-)$ is a congruence of the semigroup $(S, \cdot)$ for which $(x,y)\in \theta$ implies $(x^{+},y^{+}), (x^{-},y^{-}) \in \theta$.  For brevity we also refer to such a congruence as a \emph{$\pm$-congruence}.
\end{definition}

 \section{Examples and related classes}
 For the remainder of the paper, we shall (except where emphasis may be needed) omit the symbol $\cdot$ and instead denote multiplication in the semigroup(s) by juxtaposition.  In the Introduction we remarked on comparisons with restriction semigroups and Ehresmann semigroups, and here we expand on those comments.   
  \subsection {Restriction semigroups}   
By definition (see  Kudryavtseva \cite{Ku} for a recent paper which inspired the present work), $S$ is a \emph{restriction semigroup} if  it satisfies (in the present notation) axioms (\ref{eq:locsgp}a, c, d, f, g), the \emph{ample} identities
\begin{equation}\label{eq:conj}
 xy^{+} = (xy)^{+}x,  \hspace{3mm} x^{-}y = y(xy)^{-}  ,
\end{equation}
and has commuting projections (CP).  In effect, in comparison with restriction semigroups, localisable semigroups dispense with both (CP) and ample identities. 

Incidentally, (CP) may be replaced by the condition that $S^{+}$ be a band: 
\begin{lemma}\label{cp}
If both ample identities (\ref{eq:conj}) hold in a semigroup $S$ satisfying (\ref{eq:locsgp}a,d,g) and  $S^{+}S^{+}\subseteq S^{+}$, then (CP) holds.  
\end{lemma}
 \begin{proof}
 Let $p,q \in S^{+}$.  Then with $x= p = p^{-}$ and $ y= q = q^{+}$ in (\ref{eq:conj}) we have $pq = (pq)^{+}p = (pq)p$; with $x=q$ and $y=p$ we have $qp = p(qp)$.  Thus $pq = pqp = qp$. 
\end{proof}
In this context it is well-known that  (\ref{eq:locsgp}b) is a consequence of  (\ref{eq:conj}):
\begin{align*}
(xy^{+})^{+} &=^{5.1} &&[(xy)^{+}x]^{+} &&=^{4.1c} &&(xy)^{+}x^{+}  \\
 &   =^{CP} &&x^{+}(xy)^{+} &&=^{4.1c} &&(x^{+}xy)^{+} &=^{4.1a} (xy)^{+}.
\end{align*}
 The dual statement also holds, so any restriction semigroup, and in particular any inverse semigroup, satisfies all equations (\ref{eq:locsgp}) and so is a localisable semigroup. 
 This raises the question of examining localisable  
 semigroups with (CP). 
  \subsection{Ehresmann semigroups}
 \emph{Left Ehresmann} semigroups are unary semigroups satisfying (\ref{eq:locsgp}a), (\ref{eq:locsgp}b) and 
 \begin{equation}\label{eq:star}
  (x^{+} y^{+})^{+} = x^{+}y^{+} = y^{+} x^{+}
\end{equation}
(Branco et al.\cite{BGG}; Lawson \cite{La21}).  Right Ehresmann semigroups are defined dually, and we say $S$ is \emph{Ehresmann} if it is left and right Ehresmann and (\ref{eq:locsgp}g) holds.  Note that (\ref{eq:star}) implies (CP).  It turns out that localisable semigroups are precisely a non-commutative version of Ehresmann semigroups.  
\begin{lemma}\label{ehr}
 $(S,\cdot ,+) = S$ is left localisable and satisfies (CP) if and only if it is left Ehresmann.
\end{lemma}
\begin{proof}
Axioms (\ref{eq:locsgp}a) and b) are common to both definitions; so it is sufficient to show that, in their presence, the conjunction of (CP) and (\ref{eq:locsgp}c) is equivalent to (\ref{eq:star}).  So consider the identity $(x^{+}y)^{+} = (x^{+}y^{+})^{+}$ holding by (\ref{eq:locsgp}b). 

 If  (\ref{eq:locsgp}c) holds, the left side is $x^{+}y^{+}$ and with (CP) we have (\ref{eq:star}).  
If (\ref{eq:star}) holds, the right side of the same identity is $ x^{+}y^{+}$ and this gives (\ref{eq:locsgp}c); {\it en passant} we have (CP) also.
\end{proof}
\begin{corollary}
 $S$ is localisable and satisfies (CP) if and only if it is Ehresmann.
\end{corollary}
We consider one-sided localisable semigroups further in the next Section.  For extra developments in the topic of one-sided restriction-like semigroups, the reader may consult Stokes \cite{St3}.
 \subsection{Monoids}
 Since any monoid $S$ with identity $1$ may be regarded as a restriction semigroup, it also is a localisable semigroup.  One simply defines 
$s^{+} = s^{- } = 1$ for all $s\in S$.  
 There is just one projection, and $S$ is then called a \emph{reduced} restriction monoid.  
 (There are localisable monoids with multiple projections: for example, let $S$ be any localisable semigroup and adjoin a new identity $1$ with $1^{+} = 1^{-} = 1$.) 
 
Reduced restriction monoids occur as subsemigroups in any localisable $S$: if $e\in P$, then consider $M_{e}:=\{x\in S\colon x^{+} = x^{-} = e\}$.  
If $x,y\in M_{e}$, $(xy)^{+} = (xy^{+})^{+} = x^{+} = e$ by (\ref{eq:locsgp}b) and dually, $(xy)^{-} = e$.   So $xy\in M_{e}$ and $M_{e}$ is a subsemigroup of $S$, a localisable semigroup in its own right, and a reduced restriction monoid.  It coincides with the monoid of endomorphisms of object $e$ in $\mathscr{C}(S)$.

\subsection{Orthocryptogroups}
Any band $S$ is a localisable semigroup when we take $s^{+} = s^{-} = s$.  
More generally, let $S$ be an \emph{orthocryptogroup}, which is to say it is a union of groups, is orthodox (its idempotents form a band) and its Green's relation $\mathscr{H}$ is a congruence.  By Theorem II.8.5 (iii) of the monograph \cite{PR}, $s^{0}t^{0} = (st)^{0}$ holds for all $s,t\in S$,
where, as usual, $s^{0}$ denotes the identity of the subgroup of $S$ which contains $s$.   Then,  
on setting $s^{+} = s^{-} = s^{0}$, one immediately verifies that (\ref{eq:locsgp}a--g) all hold, and $S$ is a localisable semigroup.

Notwithstanding all the above, we still need  non-trivial and non-artificial concrete examples of  localisable semigroups.

\section{One-sided  reducts and generalised Green's relations}
In this section, we compare one-sided versions of localisable semigroups, considered as a class of unary semigroups, with other such classes.  Consider a left localisable semigroup $(S,\cdot, +)$ as defined in (\ref{eq:locsgp}a, b, c), and let a relation on $S$ be defined thus:
$$s\,\widetilde{\leq}_{R}\,t  \text{   if   } pt = t \text{  implies  } ps = s   \text{   for all  }   p\in S^{+} .$$
It is readily seen that $\widetilde{\leq}_{R}$ is reflexive and transitive, and contains the right division relation ($s = tx$ for some $x$). We record some results for further reference; note the crucial use of (\ref{eq:locsgp}c), $(ps)^{+} = ps^{+}$.
\begin{lemma}\label{TFAE}
 The following are equivalent:
 \begin{enumerate}[i)]\label{eq:swung}
 \item $s\,\widetilde{\leq}_{R}\,t $;
 \item $s^{+}\,\widetilde{\leq}_{R}\,t^{+} $;
 \item $s^{+} \leq_{R} t^{+}$;
 \item $s = t^{+}s$.
 \end{enumerate}
\end{lemma}
\begin{proof}
 Let $p\in S^{+}$.  If (i) holds, $pt^{+}=t^{+}$ implies $pt=t$ and in turn $ps=s$, and  $ps^{+} = (ps)^{+} = s^{+}$, so (ii) holds.  Also, since $t^{+}t^{+} = t^{+}$, (ii) implies $t^{+}s^{+} = s^{+}$, which is (iii).  Next, $t^{+}s^{+} = s^{+}$ implies $ t^{+}s = s$, so (iii) implies (iv).  Finally, if $t^{+}s = s$ and $pt=t$, then 
$ps = pt^{+}s = (pt)^{+}s = t^{+}s = s,$
and (i) holds.
\end{proof}
 Symmetrising $\widetilde{\leq}_{R}$ gives the `swung' Green's relation 
 $\widetilde{\mathscr{R}} = \{(s,t)\colon ps=s \Leftrightarrow pt=t\}$. 
 \begin{corollary}
 For all $s,t,u\in S$,
 $s\,\widetilde{\leq}_{R}\,t$ implies $us\,\widetilde{\leq}_{R}\,ut$, and so $\widetilde{\mathscr{R}}$ is a left congruence.  
\end{corollary} 
 \begin{proof}
 Suppose $s\,\widetilde{\leq}_{R}\,t$; by part (iv) of the Lemma, $t^{+}s = s$.  Consider 
 $(ut)^{+}us = (ut^{+})^{+}us = (ut^{+})^{+}ut^{+}s = (ut^{+})s = us$; this shows $us\,\widetilde{\leq}_{R}\,ut$, and the rest follows.
\end{proof}
Let $E$ be a distinguished set of idempotents of a unary semigroup $(S,\cdot,+)$, and recall that $S$ is called \emph{weakly left $E$-abundant} (Gould) or \emph{left $E$-semiabundant} (Stokes) if every
$\widetilde{\mathscr{R}}$-class contains an element of $E$.  
It follows that left localisable semigroups are weakly left $E$-abundant with $E = \{s^{+} \colon s\in S\} = S^{+}$.  
However, the converse is false:

{\bf Example.}  Let $S$ be the semigroup with zero $0$ generated by idempotents $e, f$ such that $fe = 0$, writing $ef = a$ for convenience.  The diligent reader will verify with a little light work that $S = \{e, f, a, 0\}$,   that $E = E(S) = \{e,f,0\}$, and that the assignment 
$e^{+}=e=a^{+}, f^{+}=f, 0^{+}=0$ satisfies  $y^{+}x = x$ if and only if 
$y^{+}x^{+} =x^{+}$.
Thus $x\,\widetilde{\mathscr{R}}\, x^{+}$ for all $x\in S$, and so $S$ is weakly left $E$-abundant.  
However the corresponding
unary operation $x \mapsto x^{+}$ does not make $(S,\cdot , +)$  left localisable, for $E$ is not a band ($ef = a \notin E$).  

{\bf On generalised D-semigroups.}  A more stringent condition requires a \emph{unique} element of $E$ in each $\widetilde{\mathscr{R}}$-class; this gives the 
 \emph{generalised D-semigroups} studied by Stokes \cite{St2}.  This class neither contains nor is contained by that of left localisable semigroups.  Indeed, a generalised D-semigroup $S$ in which $S^{+}$ is not a band cannot be left localisable.  On the other hand, a band $S$ having an $\widetilde{\mathscr{R}} =  {\mathscr{R}}$-class with distinct elements fails the uniqueness requirement, so $S$ is left localisable (with $s^{+}= s$) but not a generalised D-semigroup.

The intersection of these two classes merits consideration.  Let $S$ be a left localisable semigroup in which $S^{+}$ is a left regular band.  Conditions 1, 2 and 3 of \cite{St2} follow from (\ref{eq:locsgp}a,b,c) and condition 4 holds by left regularity of $S^{+}$.  (The notations $x^{+}$ and $D(x)$ are identified.)  Now by Corollary 2.13 of \cite{St2}, $S$ is a D-semigroup with band of projections, necessarily left regular.  So $S$ also satisfies condition (D3) of \cite{St1}, namely $(xy)^{+}x^{+}=x^{+}(xy)^{+}=(xy)^{+}$; the second of these equations is a consequence of  (\ref{eq:locsgp}c) and the first implies that $S^{+}$ is a left regular band.  So the class of left localisable generalised D-semigroups is identified as the class of D-semigroups with a band of projections or, equivalently, the class of left localisable semigroups satisfying $(xy)^{+}x^{+} = x^{+}(xy)^{+}$. 
   
 Note, however, that the unary projection map in an arbitrary left localisable semigroup $(S,\cdot,+)$ may be modified to produce a generalised D-semigroup, as follows. Choose $X\subseteq S^{+}$ to be a cross-section of the equivalence $\widetilde{\mathscr{R}}$, and for $s\in S$, define $s^{\oplus}$ to satisfy $s^{\oplus} \in X$ and $s^{+}\,\widetilde{\mathscr{R}}\,s^{\oplus}$.  It follows that $s^{\oplus}s^{\oplus} = s^{\oplus}$ and $s^{\oplus}s^{+} = s^{+}$, whence  
$s^{\oplus} s = s^{\oplus}s^{+}s = s^{+}s = s$.  That is, $(S,\cdot,\oplus)$ is a generalised D-semigroup.  

\section{Other structural features} 
\subsection*{An order}  We may define an analogue of the natural partial order on a restriction semigroup:
\begin{definition}\label{po}
 Write $s \unlhd t$ if $s = s^{+}t = ts^{-}$.
\end{definition}
Recall that the natural (Mitsch) partial order $\leq_{M}$ on any semigroup $S$ is defined by 
$s\leq_{M}t$ if $s = t$ or there exist $x,y \in S$ such that $s = xt = ty = xty$. 
\begin{lemma} 
The relation $\unlhd$ is a partial order on $S$ which is contained in the  natural partial order $\leq_{M}$, and coincides with it when restricted to the band  $S^{+}$. 
\end{lemma}
\begin{proof}
From Definition \ref{po} it is clear that $s\unlhd s$ for all $s\in S$.  If $s \unlhd t\unlhd u$, then $s = s^{+}t = s^{+} t^{+}u = (s^{+} t)^{+}u$ (by (\ref{eq:locsgp})c) $= s^{+}u$, and dually $s=us^{-}$, whence $s\unlhd u$. Then if $s\unlhd t\unlhd s$ we have $s = ts^{-}$ and $t = t^{+}s = t^{+}ts^{-} = ts^{-} = s$.   
It is immediate from the definitions that $s\unlhd t$  implies $s \leq_{M} t$.   When $s= s^{+} = s^{-}$, we have $s\unlhd t$ if and only if $s=st = ts$, as claimed. 
\end{proof}
The order $\unlhd$ is the identity on any reduced restriction monoid, and so differs in general from the Mitsch order.
\subsection*{Actions and transformation representations}  The transcription structures lead to actions of $S$ on the projections.  These give representations of $S$ by transformations, which  naturally have weaker properties than their counterparts in the restriction semigroup setting.  The development here may be compared to that of El-Qallali et al. \cite{QFG} and Gomes and Gould \cite{GG}, which can go much deeper because of the extra conditions. 

\begin{definition}\label{action}
 For $p\in S^{+}$ and $s,t \in S$, let us write $p^{s}$ for $(ps)^{-}\in S^{+}$ and dually, $^{t}p$ for $(tp)^{+}$.  
\end{definition}
The ample identities provide a notion of conjugacy (cf. Ara\'{u}jo et al. \cite{AKKM}) and although those are not available here, the notations $p^{s}$ and $^{t}p$ conveniently and compactly reflect the fact that there is structure enough to yield actions. 
\begin{lemma} \label{ishom}
For $p\in S^{+}$ and $s,t \in S$, $(p^{s})^{t} =  p^{st}$ and $^{s}(^{t}p) = \,^{st}p$.  
\end{lemma}
\begin{proof}
 $(p^{s})^{t} = (p^{s}t)^{-} = ((ps)^{-} t)^{-} = ((ps) t)^{-}$
 from axiom (\ref{eq:locsgp}e), and so $(p^{s})^{t} = (p(st))^{-} = p^{st}$.  The second equation is dual.
 \end{proof}
 Thus we have defined a (right) action $S^{+}\times S \to S^{+}$, $(p,s)\mapsto p^{s}$.  
 Let us write $s\delta \colon p\mapsto p^{s}$ for $p\in S^{+}$, and from this define the map  $\delta\colon s\mapsto s\delta$.   
By Lemma \ref{ishom}, $(st)\delta = s\delta\circ t\delta$, so $\delta$ is a semigroup morphism and a representation of $S$ (as a plain semigroup) in the transformation monoid $\mathscr{T}_{S^{+}}$. 

There is dually a left action $(s,p)\mapsto \; ^{s}p  = (sp)^{+}$ and corresponding (anti-) representation $\gamma$.
Notationally it is most convenient to write $\gamma$ and $\gamma s$ as left mappings, $\gamma\colon s\mapsto \gamma s$, where $\gamma s(p) = \,^{s}p = (sp)^{+}$,  
so that $\gamma st(p) =\,^{st}p = (stp)^{+} = \gamma s \circ \gamma t(p)$, and $\gamma$ manifests as a representation 
$S\rightarrow \mathscr{T}_{S^{+}}$.   
Thus the combination of $\gamma$ and $\delta$ results in a representation 
 $\gamma\otimes\delta\colon 
 S\rightarrow \mathscr{T}_{S^{+}}\times\mathscr{T}_{S^{+}}$,
 $s\mapsto (\gamma s, s\delta)$. 
  
 Let us note that $\gamma s^{\pm}(p) = (s^{\pm}p)^{+} = s^{\pm}p$, and dually $(p)s^{\pm}\delta = ps^{\pm}$.  
 (Here and later we use $\pm$ with the meaning that it may be replaced in a statement by either $+$ throughout or by $-$ throughout.)
 Thus $\gamma s^{\pm} = \lambda_{s^{\pm}}$, the inner left translation on $S^{+}$, and $s^{\pm}\delta = \rho_{s^{\pm}}$, the inner right translation on $S^{+}$.  
 As usual, we denote the band of inner left [right] translations of $S^{+}$ by $\Lambda_{0}$ [$\Rho_{0}$], and the band of linked pairs of inner translations (bitranslations) $(\lambda_{p},\rho_{p})$ of $S^{+}$ by $\Omega_{0}$.   Recall that $S^{+}\cong \Omega_{0} \leq \Lambda_{0} \times{\Rho}_{0}$ via the map $p\mapsto (\lambda_{p},\rho_{p})$, so that we may identify $S^{+}$ with $\Omega_{0}$.

The point of these remarks is that the definition of projections on the semigroup $S(\gamma\otimes\delta)$ by means of 
$(\gamma s, s\delta)^{\pm} = (\gamma s^{\pm}, s^{\pm}\delta)$ is well-defined if and only if $(\gamma s, s\delta) = (\gamma t, t\delta)$ implies $(\gamma s^{\pm}, s^{\pm}\delta) = (\gamma t^{\pm}, t^{\pm}\delta)$.  By the preceding discussion, this is so if and only if $(\lambda_{s^{\pm}},\rho_{s^{\pm}}) = (\lambda_{t^{\pm}},\rho_{t^{\pm}})$, i.e., $s^{\pm} = t^{\pm}$.  
This leads us to consider the relation 
 $\mu  =  \mu^{S}$ such that $(s,t)\in \mu$ if and only if 
\begin{equation}\label{ker}
s^{+} = t^{+}, s^{-} = t^{-}, (sp)^{+} = (tp)^{+}\text{   and   } (ps)^{-} = (pt)^{-} \text{   for all   } p\in S^{+}.
\end{equation}
\begin{remark}
 We may write (\ref{ker}) as 
 $$(s,t)\in \mu \text{   if and only if    }  (ps)^{\pm} =  (pt)^{\pm}   \text{   and   } (sp)^{\pm} = (tp)^{\pm}  \text{   for all   } p\in S^{+}$$
 since $(ps)^{+} = ps^{+} = p\rho_{s^{+}}$ by (\ref{eq:locsgp}b), etc.
\end{remark}
Now we define a relation $\theta$ on $S$ to be \emph{projection-separating} if $p,q\in S^{+}$ and $(p,q)\in \theta$ imply that $p = q$.  Clearly, $\mu$ is projection-separating.  It may not be a congruence, but we do have the following.
\begin{lemma}
Let $\theta$ be a $\pm$-congruence on $S$.  Then $\theta$ is projection-separating if and only if $\theta \subseteq \mu$.
\end{lemma}
\begin{proof}
If  $(s,t)\in \theta$, then 
$(s^{\pm},t^{\pm})\in \theta$, and since $\theta$ is projection-separating, $s^{\pm} = t^{\pm}$.  Moreover, 
$((sp)^{\pm},(tp)^{\pm})\in \theta$ for any $p\in S^{+}$, whence $(sp)^{\pm} = (tp)^{\pm}$.  Similarly $(ps)^{\pm} = (pt)^{\pm}$, and $(s,t)\in \mu$ ensues by (\ref{ker}).

Conversely, if $\theta \subseteq \mu$ and $(s^{+},t^{+})\in \theta$, then $(s^{+},t^{+}) \in \mu$, and by appropriate choices of $p$ in equation (\ref{ker}) we have 
$ s^{+}s^{+} = t^{+}s^{+} = t^{+}t^{+}, $
whence $s^{+} = t^{+}$ and $\theta$ is projection-separating.  
\end{proof}

By analogy with the theory of inverse semigroups, we  say that the localisable semigroup $T$ is \emph{fundamental} if $\mu^{T}$ is the identity relation.  

 \section{The groupoid case: $\ast$-localisable semigroups}
Recall that a category $C$ is a \emph{groupoid} if each morphism $x$ has an inverse in $C$, denoted by $x^{-1}$, such that $x\circ x^{-1} = x^{+}$ and $x^{-1}\circ x = x^{-}$.  (This is equivalent to $C$ being cancellative and regular in the von Neumann sense.)
\begin{definition}
A \emph{transcription groupoid} is a groupoid which is also a transcription category. 
\end{definition}
We shall provide a description of the corresponding transcription semigroup $\mathscr{S}(C)$, and to this end make the following definitions.
\begin{definitions}

$(S,\cdot,{}^*)$ is a \emph{regular unary semigroup} if 
it satisfies 
\begin{subequations}\label{eq:starloc}
\begin{align}
  x x^* x &= x   \hspace{3mm}   \text{    and    }			\label{eq:starloc a} \\
  x^*xx^* &= x^*          \label{eq:starloc b}
\end{align}  
{for all }  $x\in S$.   A regular unary semigroup $(S,\cdot,{}^*)$ is \emph{$\ast$-localisable} if it also satisfies 
\begin{align}
  (x(xx)^*x)^* &= x(xx)^*x ,	\label{eq:starloc c} \\
  x(xyy^*)^* &= xy(xy)^*   \hspace{3mm}   \text{    and    } \label{eq:starloc d} \\
  (x^*xy)^*y &= (xy)^*xy  .	\label{eq:starloc e}
\end{align}
\end{subequations}
\end{definitions}

The following observation is well known.
\begin{lemma}\label{Lem:idem_fix}
Let $(S,\cdot,{}^*)$ be a regular unary semigroup. Then ${}^*$ fixes every idempotent
of $S$ if and only if \eqref{eq:starloc c} holds.
\end{lemma}
\begin{proof}
  If ${}^*$ fixes every idempotent, then it certainly fixes $x(xx)^*x$ since \eqref{eq:starloc b} implies
  this is an idempotent for every $x\in S$. Conversely, if \eqref{eq:starloc c} holds and if $e^2 = e$, then
\[
    e^* \overset{\eqref{eq:starloc a}}{=} (ee^*e)^* = (e(ee)^*e)^*
    \overset{\eqref{eq:starloc c}}{=} e(ee)^*e = ee^*e \overset{\eqref{eq:starloc a}}{=} e\,. \qedhere
\]
\end{proof}

\begin{lemma}\label{Lem:*-fix}
Let $(S,\cdot,{}^*)$ be a 
$\ast$-localisable semigroup. 
Then (1) the identity
  \begin{equation}\label{eq:starloc b'}
    x^{**} = x
  \end{equation}
  holds, and (2) ${}^*$ fixes every $e\in S$ satisfying
  $(e^*)^2 = e^*$.
\end{lemma}
\begin{proof}
(1) First we show
 \begin{subequations}\label{eq:ss}
 \begin{align}
  x^* x^{**} &= x^* x       \label{eq:ssa} \\
  x^{**} x^* &= x x^*\,.    \label{eq:ssb}
  \end{align}
\end{subequations}
For \eqref{eq:ssa}, we compute
\[
x^* x^{**} \overset{\eqref{eq:starloc b}}{=} x^* (x^*xx^*)^*
\overset{\eqref{eq:starloc d}}{=} x^* x (x^* x)^*
= x^* x x^* x \overset{\eqref{eq:starloc a}}{=} x^* x\,,
\]
using Lemma \ref{Lem:idem_fix} in the third equality, and \eqref{eq:ssb} is proved dually. Thus
\[
x^{**} \overset{\eqref{eq:starloc b}}{=} x^{**}x^*x^{**}
\overset{\eqref{eq:ssa}}{=} x^{**} x^* x
\overset{\eqref{eq:ssb}}{=} x x^* x
\overset{\eqref{eq:starloc a}}{=} x\,,
\]
which is \eqref{eq:starloc b'}.

(2) If $(e^*)^2 = e^*$, then Lemma \ref{Lem:idem_fix} gives $e^{**} = e^*$ and so
$e^* = e$ by \eqref{eq:starloc b'}.
\end{proof}

Since \eqref{eq:starloc a} and \eqref{eq:starloc b'} evidently imply \eqref{eq:starloc b},
we could equally well have used \eqref{eq:starloc b'} as an axiom in place of \eqref{eq:starloc b}.
Now we can justify the terminology `$\ast$-localisable'.
\begin{theorem}\label{slisl}
  Let $(S,\cdot,{}^*)$ be a $*$-localisable semigroup and define unary operations ${}^+,{}^- :S\to S$ by
  $x^+ = xx^*$ and $x^- = x^*x$ for all $x\in S$. Then $(S,\cdot,{}^+,{}^-)$ is a localisable semigroup.

  Conversely, let $(S,\cdot,{}^+,{}^-)$ be a localisable semigroup, assume that $S$ is regular and
  suppose there exists an inverse mapping ${}^* :S\to S;x\mapsto x^*$ such that $x^+ = xx^*$ and
  $x^- = x^*x$ for all $x\in S$. Then $(S,\cdot,{}^*)$ is $*$-localisable.
\end{theorem}
\begin{proof}
 ($\Rightarrow$) From \eqref{eq:starloc a}, we immediately have $x^+x = x = xx^-$, $x^+ x^+ = x^+$,
 $x^- x^- = x^-$. From \eqref{eq:starloc b}, we also get $x^* x^+ = x^*$ and $x^- x^* = x^*$.

 In \eqref{eq:starloc d}, replace $x$ with $xy^+$. The left hand side becomes
 $xy^+ (xy^+ y^+)^* = (xy^+)^+$. The right side becomes $xy^+ y(xy^+y)^* = (xy)^+$. Therefore
 \begin{equation}\label{eq:loc-tmp1}
 (xy^+)^+ = (xy)^+\,.
 \end{equation}
 A dual argument using \eqref{eq:starloc e} gives $(x^- y)^- = (xy)^-$.

 Next we have
 \begin{equation}\label{eq:*-tmp1}
    (x^+ y)^+ y = x^+ y\,.
 \end{equation}
 Indeed,
 \[
 (x^+ y)^+ y = x^+ y (x^+ y)^* y = x^+ y ((x^+)^- y)^* y
 \overset{\eqref{eq:starloc e}}{=} x^+ y (x^+ y)^* x^+ y
 \overset{\eqref{eq:starloc a}}{=} x^+ y\,.
 \]
 Now we compute
 \begin{align*}
   (x^+ y^+)^* (x^+ y^+)^* &\overset{\eqref{eq:*-tmp1}}{=} (x^+ y^+)^* ((x^+ y^+)^+ y^+)^*  \\
    &= (x^+ y^+)^* (x^+ y^+)^+ ((x^+ y^+)^+ y^+)^* \\
    &\overset{\eqref{eq:starloc d}}{=} (x^+ y^+)^* ((x^+ y^+)^+ y)^+ \\
    &\overset{\eqref{eq:*-tmp1}}{=} (x^+ y^+)^* (x^+ y)^+ \\
    &\overset{\eqref{eq:loc-tmp1}}{=} (x^+ y^+)^* (x^+ y^+)^+ \\
    &= (x^+ y^+)^*\,.
 \end{align*}
Thus by Lemma \ref{Lem:*-fix}, we obtain
\begin{align}
  (x^+ y^+)^* &= x^+ y^+\,, \label{eq:*-tmp2}\\
  \intertext{hence}
  x^+ y^+ x^+ y^+ &= x^+ y^+\,. \label{eq:*-tmp3}
\end{align}
Finally,
\[
(x^+ y)^+ \overset{\eqref{eq:loc-tmp1}}{=} (x^+ y^+)^+ = x^+ y^+ (x^+ y^+)^*
\overset{\eqref{eq:*-tmp2}}{=} x^+ y^+ x^+ y^+
\overset{\eqref{eq:*-tmp3}}{=} x^+ y^+\,.
\]
A dual argument gives $(x y^-)^- = x^- y^-$.

\medskip

For the converse, we have
\[
(xy)^+ = (xy^+)^+ = xy^+(xy^+)^* = xx^- y^+(xy^+)^*
= x(xy^+)^- (xy^+)^* = x(xy^+)^*\,,
\]
which is \eqref{eq:starloc d}, and \eqref{eq:starloc e} is proved dually. By Lemma \ref{Lem:*-fix},
we also have \eqref{eq:starloc b'}.

Finally we show that ${}^*$ fixes every idempotent. We first do this for elements of $S^+$:
\[
(x^+)^* = (x^+)^* x^{++} = (x^+)^* x^+ = (x^+)^- = x^+.
\]
Now assume $e^2 = e$. Then $e^+ = ee^* = eee^* = ee^+$, and so
\[
e^* = e^* e^+ = e^* e e^+ = e^- e^+ = (ee^+)^- = (e^+)^- = e^+\,.
\]
This gives $e \overset{\eqref{eq:starloc b'}}{=} e^{**} = (e^+)^* = e^+ = e^*$.
By Lemma \ref{Lem:idem_fix}, \eqref{eq:starloc c} holds, completing the proof.
\end{proof}

Hence when we consider $(S,\cdot,^{\ast})$ as a localisable semigroup, we understand that its projections are given by $x^{+} = xx^{\ast}$ and $x^{-} = x^{\ast}x$. 

Note that for the converse in Theorem \ref{slisl}, one must assume some compatibility between an inverse mapping
and the unary operations defining localisability; regularity of the semigroup is not sufficient.
Consider, for instance, the $2$-element semilattice $S = \{0,1\}$ with $0 < 1$, which has the structure of a localisable semigroup when $x^+ = x^- = 1$ for $x =0,1$.  However this does not derive from
the only available inverse mapping, which is $x^* = x$; for $0^{+} = 1$ 
but $00^{*} = 0$.
\begin{corollary}\label{tg=op}
$(S, \cdot,\ast)$ is a $\ast$-localisable semigroup if and only if  $\mathscr{C}(S)$ is a transcription groupoid. 
\end{corollary}
\begin{proof}
Let $S$ be a $\ast$-localisable semigroup.  By Theorem \ref{th2},  $C=\mathscr{C}(S)$ is a transcription category.
 For any $x\in C$, $x\circ x^{\ast}$ is defined and equals  $xx^{\ast} = x^{+}$, and similarly $ x^{\ast}\circ x =x^{-}$, showing that $C$ is a groupoid.
Conversely, let $C$ be a transcription groupoid; by Theorem \ref{th1}, $\mathscr{S}(C)$ is a localisable semigroup. Since it is also a groupoid, the map $x\mapsto x^{-1}$ satisfies the conditions for the converse part of Theorem \ref{slisl} and so $S$ is $\ast$-localisable.  
\end{proof}
The following result is Theorem 3 in \cite{Fi}; for completeness, we give a short proof using the notation of the present work.
\begin{proposition}\label{ECES}
Let $S$ be a $\ast$-localisable semigroup.  Then every idempotent is a projection; in fact, for $x\in S$,  $x^{2} = x$ implies $x = x^{+} = x^{-}$.  Consequently, $S$ is orthodox.
\end{proposition}
\begin{proof}
If $x^{2} = x$ then $xxx^{\ast} = xx^{\ast}$ and so $xx^{+} = x^{+}$.  Then 
$x^{+} = (x^{+})^{-} = (xx^{+})^{-} = x^{-}x^{+}$, by (\ref{eq:locsgp}c) with $y=x^{+}$.   Dually or similarly, $x^{-} = x^{-}x^{+}$; and then $x^{+} = x^{-}$ and $x = xx^{-} = xx^{+} = x^{+}$ as already seen.  
\end{proof}
\subsection*{Remarks}
\begin{enumerate}
\item Proposition \ref{ECES} contrasts with the general situation: if $C$ is merely a category, 
$E(C) = \{x\in E(S)\colon x^{+} = x^{-}\} .$  
(As usual, we denote the set of idempotents in $S$ or 
$C$ by $E(S)$ or $E(C)$ respectively.) For let $x \in S$ with $x=x^2$ and $x^{+} = x^{-}$.  Then by Definition \ref{def:circ},  $x\circ x$ is defined in $C$ and $x\circ x = x$ in $\mathscr{C}(S) = C$.  In the converse direction, let  $x\circ x = x$ in $C$. Then $x^{+} = x^{-}$ and $x = x\otimes x = x^2$.

\item The regular involution $^{\ast}$ is not necessarily an anti-automorphism of $S$. 
For example, let $S$ be a band and set $x^{\ast} = x =x^{+} = x^{-}$  for each $x\in S$; it is easily verified that $C = \mathscr{C}(S)$ is a transcription groupoid.  If $S$ is not a semilattice then there are  $x, y\in S$ such that $xy \neq yx$, which is equivalent to $(xy)^{\ast} \neq y^{\ast} x^{\ast}$.
\end{enumerate}
\section*{Acknowledgements}
 The authors thank a referee and the editor for the improvements resulting from (respectively) their gently pointing out a nonsensical passage and suggesting additional references.  For sharing some of their work before publication, they thank Mark Lawson and Tim Stokes, who also made helpful comments on drafts of this work including detecting errors.


\begin{thebibliography}{99}
\bibitem {AKKM}  Ara\'{u}jo, J., Kinyon,  M. K., Konieczny, J.,  Malheiro, A.: Four notions of conjugacy for abstract semigroups. { Proc. Roy. Soc. Edinburgh Sect. A.} {\bf147}, 1169--1214 (2017)
\bibitem{BGG} Branco, M. J. J.,  Gomes, M. S., Gould, V.: Ehresmann monoids.  J. Algebra {\bf 443} 349--382 (2015) \textsf{http://dx.doi.org/10.1016/j.jalgebra.2015.06.035}
\bibitem{CoLa} Cockett, J. R. B., Lack, S.: Restriction categories I. Categories of partial maps. { Theoret. Comput. Sci.}  {\bf 270}, 223--259 (2002)
\bibitem{Q} El-Qallali, A.:  Congruences on ample semigroups. { Semigroup Forum}  {\bf 99}, 607--631 (2019) \textsf{  https://doi.org/10.1007/s00233-018-9988-4}
\bibitem{QFG} El-Qallali, A., Fountain, J., Gould, V.: Fundamental representations for classes of semigroups containing a band of idempotents. Comm. Algebra {\bf 36}, 2998--3031 (2008)
\bibitem{Fi} FitzGerald, D. G.:  Groupoids on a skew lattice of objects.  { Art Discr. Appl. Math.} \# P2.03  (2019) \textsf{
 https://doi.org/10.26493/2590-9770.1342.109}
\bibitem {GG} Gomes, G.M.S., Gould, V.: Fundamental semigroups having a band of idempotents. Semigroup Forum {\bf 77}, 279--299 (2008) \textsf{ https://doi.org/10.1007/s00233-007-9041-5}
\bibitem{GoN} Gould, V.: Notes on restriction semigroups.    \textsf{http://www-users.york.ac.uk/~varg1/restriction.pdf}   (2010). Accessed 22 August 2020
\bibitem{GoSt}  Gould, V., Stokes, T.:  Constellations and their relationship with categories. { Algebra Universalis}  {\bf 77}, 271--304 (2017)
\bibitem{Ho} Hollings, C.:  From right PP monoids to restriction semigroups: a survey. { Eur. J.  Pure Appl. Math.}  {\bf  2},  21--57  (2009)
\bibitem{Ho2} Hollings, C.: The Ehresmann-Schein-Nambooripad theorem and its successors. Eur. J. Pure Appl. Math. {\bf 5},  414--450 (2012)
\bibitem{Jo} Jones, P. R.:  A common framework for restriction semigroups and regular $\ast$-semigroups. { J. Pure Appl. Algebra}   {\bf 216}, 618--632  (2012)
\bibitem{Ku} Kudryavtseva, G.:  Two-sided expansions of monoids. 
 { Internat. J. Algebra Comput.} {\bf 29}, 1467--1498 (2019)
\bibitem{LaIS} Lawson, M. V.  { Inverse Semigroups: The Theory of Partial Symmetries}.  World Scientific, River Edge, NJ (1998)
\bibitem{La91} Lawson, M. V.: Semigroups and ordered categories I. The reduced case. { J. Algebra}  {\bf 141},  422--462 (1991)
\bibitem{La21} Lawson, M. V.: On Ehresmann semigroups. Semigroup Forum (2021) 
\textsf{https://doi.org/10.1007/s00233-021-10200-22021}
\bibitem{Ma} Malandro, M.E.:  Enumeration of finite inverse semigroups. { Semigroup Forum}  {\bf 99}, 679--729 (2019)
\textsf{https://doi.org/10.1007/s00233-019-10054-9}
\bibitem{MiCl} Miller, D. D., Clifford, A. H.: Regular $\mathscr{D}$-classes in semigroups. { Trans. Amer. Math. Soc.}  {\bf  82}, 270--280  (1956)
\bibitem{NS} Nordahl, T. E., Scheiblich, H. E.: Regular $\ast$-semigroups. { Semigroup Forum}  {\bf 16}, 369--377 (1978)
\bibitem{PR} Petrich, M., Reilly, N. R.:  { Completely Regular Semigroups},
Canadian Mathematical Society Series of Monographs and Advanced Texts, Vol. 23. Wiley, New York, NY (1999)
\bibitem{St1} Stokes, T.:  D-semigroups and constellations. { Semigroup Forum}  {\bf 94}, 442--462 (2017) 
\bibitem{St2}  Stokes, T.:  Generalised domain and E-inverse semigroups.  {Semigroup Forum}  {\bf 97}, 32--52 (2018)
\bibitem{St3}Stokes, T.:  How to generalise demonic composition. Semigroup Forum {\bf 102}, 28--314 (2021) \textsf{https://doi.org/10.1007/s00233-020-10117-2} 
\bibitem{Sz} Szendrei, M.B.:  Structure theory of regular semigroups. { Semigroup Forum}  {\bf 100},  119--140 (2020) \textsf{https://doi.org/10.1007/s00233-019-10055-8}
\bibitem{Wa} Wang, Y.: Weakly $B$-orthodox semigroups.
Period. Math. Hung. {\bf 68}, 13--38  (2014) 
\textsf{https://doi.org/10.1007/s10998-014-0023-6}
\end{thebibliography}
\end{document}